\documentclass{article}

\usepackage[utf8]{inputenc} % Can't be put after biblatex

\author{
    Jean \textsc{Vereecke}%
    \thanks{Univ. Grenoble Alpes, CNRS, IF, 38000 Grenoble, France.\hfill\href{mailto:jean.vereecke@univ-grenoble-alpes.fr}{\texttt{jean.vereecke@univ-grenoble-alpes.fr}}}
}
\title{Uncountably many extremal Series--Sinai states for the Ising model on Lobachevsky lattices}

\makeatletter
\renewcommand\subparagraph{
    \@startsection{subparagraph}{5}{\z@}{3.25ex \@plus1ex \@minus .2ex}{-1em}{\normalfont\normalsize\bfseries}
}
\makeatother

\newcommand{\metric}[3]{\ifthenelse{\isempty{#2}}{#1}{#1\!\left(#2,#3\right)}}
\newcommand{\distance}[2]{\metric{\ud}{#1}{#2}}
\newcommand{\hausdorffDistance}[2]{\metric{\ud_\mathrm{H}}{#1}{#2}}
\newcommand{\hyperbolicDistance}[2]{\metric{\ud_\mathrm{hyp}}{#1}{#2}}
\newcommand{\abs}[1]{\left\lvert#1\right\rvert} % absolute value
\newcommand{\bfi}{\mathbf{i}} % imaginary unit
\newcommand{\cardinality}[1]{\left\lvert#1\right\rvert}
\newcommand{\exG}{\mathrm{ex}\mathcal{G}} % extremal Gibbs measures
\newcommand{\fLpq}[2]{\mathcal{L}_{#1,#2}}
\newcommand{\Hyp}[1]{{\mathbb{H}_{#1} }} % hyperbolic space
\newcommand{\ICpq}{\mathrm{IC}_{p,q}} % Lpq isoperimetric constant
\newcommand{\imPart}[1]{\mathrm{Im}\left(#1\right)} % imaginary part
\newcommand{\Lpq}{\fLpq{p}{q}} % Lpq
\newcommand{\origin}{\mathbf{o}} % reference vertex
\newcommand{\rePart}[1]{\mathrm{Re}\left(#1\right)} % real part
\let\setminus\backslash
\newcommand{\ud}{\mathrm{d}} % d vertical
\newcommand{\colonSet}[2]{\left\{#1:\,#2\right\}}
\newcommand{\commaSet}[2]{\left\{#1,\,#2\right\}}
\newcommand{\eg}{\emph{e.g.}}
\newcommand{\ie}{\emph{i.e.}}
\newcommand{\textiff}{\emph{if and only if}}
\newcommand{\resp}{resp.}

\usepackage{amsfonts}
\usepackage{amsmath}

\DeclareMathOperator*{\arcosh}{arcosh}

\DeclareMathOperator*{\artanh}{artanh}

\usepackage{amssymb}
\usepackage{amsthm}

\newtheoremstyle{special}{5pt}{5pt}{\itshape}{}{\bfseries}{.}{5pt}{\thmname{#1}~\thmnumber{#2}{\mdseries\textsc{\thmnote{ (#3)}}}}
\theoremstyle{special}
\newtheorem{generic}{Generic}
\newtheorem{conjecture}[generic]{Conjecture}

\newtheorem{definition}[generic]{Definition}
\newtheorem{lemma}[generic]{Lemma}
\newtheorem{proposition}[generic]{Proposition}
\newtheorem{remark}[generic]{Remark}
\newtheorem{theorem}[generic]{Theorem}
\usepackage[
    backend=biber,
    defernumbers=true
]{biblatex}

\usepackage[
    a4paper,
    margin=2.5cm
]{geometry}

\usepackage{hyperref}
\usepackage{overpic}
\usepackage{parskip}
\usepackage{pict2e} % Polygon in overpic
\usepackage{pgffor}
\usepackage{stmaryrd} % Package for llbracket
\usepackage{subcaption}
\usepackage{svg} % Import for .svg files

\usepackage{tikz}
\usetikzlibrary{decorations.pathreplacing}

\usepackage{wrapfig}
\usepackage{xifthen}
\usepackage{xcolor}

\definecolor{darkgreen}{RGB}{0, 138, 13}

\begin{document}

\maketitle

\begin{abstract}
    \noindent
    We construct an uncountable family of extremal Gibbs states of the low temperature Ising model on hyperbolic lattices embedded in the hyperbolic plane \(\Hyp{2}\) whose interfaces are complete geodesics of \(\Hyp{2}\).
    These states are extracted from the states constructed by D'Achille, Coquille and Le Ny in~\cite{dAClN25} by considering path in the dual lattice at close enough distance from geodesics of \(\Hyp{2}\) thanks to the Morse--Mostow lemma.

    \bigskip

    \footnotesize
    \noindent\emph{AMS 2000 subject classification}:
    Primary: 60K35~;
    Secondary: 82B20 \& 30F45.
    \noindent\emph{Keywords and phrases}:
    Ising model,
    extremal Gibbs states,
    hyperbolic lattices,
    Lobachevsky plane.
\end{abstract}

\section{Introduction}

Lund, Rasetti and Regge~\cite{LRR77} studied the ferromagnetic nearest-neighbour Ising and dimer models on lattices embedded in the Lobachevsky plane, which are invariant under the action of discrete subgroups of its isometries.
Since then, statistical mechanics models defined on hyperbolic lattices have been of interest in the physics litterature, see \eg~\cite{RNO92} or~\cite{BPR20}.

In the Euclidean case, results due to Aizenmann~\cite{Aiz80} and Higuchi~\cite{Hig79} state that on \(\mathbb{Z}^2\) any Gibbs measures of the nearest-neighbour Ising model can be expressed as a convex combination of the pure phases \(\mu^+\) and \(\mu^-\).
On \(\mathbb Z^3\), the result of Dobrushin~\cite{Dob73} exhibits a countable family of extremal non-translation invariant states at low temperature, with a localised interface at a given height~\cite{Dob74}.
For \(d \ge{3}\), the set of extremal Ising states is conjectured to be at most countable but this is difficult to prove.

Concerning the Ising model on a \(d\)-regular tree, for \(d \ge{3}\), Higuchi in~\cite{Hig77} and Bleher and Ganikhodjaev in~\cite{BH13} proved the existence of uncountably many extremal non tree-automorphism invariant ``interface states'' at low temperature.
Rozikov and Rakhmatullaev~\cite{RR08} exhibited the so-called ``weakly periodic'' Gibbs measures, corresponding to subgroups in the representation group of the Bethe lattice.
In Gandolfo--Ruiz--Shlosman~\cite{GRS12}, extremal inhomogeneous Gibbs measures at low temperature were exhibited, based upon perturbations of ground-states, that is configurations with sparse enough set of disagreement edges.
Coquille, Külske and Le~Ny~\cite{CKlN1} generalised the statements of~\cite{GRS14} to a large family of ferromagnetic finite-spin models and \(p\)-SOS models.

In hyperbolic geometry, Series and Sinai~\cite{SS90} exhibited an uncountable family of mutually singular Gibbs measures for the Ising model on the vertices of the Cayley graph of any finitely generated co-compact group of isometries of \(\Hyp{2}\), such as the hyperbolic tesselations, indexed by the geodesics of \(\Hyp{2}\), thus proving that the set of extremal measures is uncountable at low temperature.
Those states are analogous to Dobrushin states \cite{Dob73}, \ie~weak limits of finite volume states with \(\pm\) boundary condition, and we will call these states \emph{Series--Sinai states} in the rest of the paper.
The authors conjectured these states to be extremal Gibbs measures for the Ising model at sufficiently low temperature.
Later on, Gandolfo, Ruiz and Shlosman~\cite{GRS14} constructed an uncountable family of inhomogeneous Gibbs measures called the \emph{Millefeuilles}, having infinitely many interfaces with positive density.
More recently, D'Achille, Coquille and Le~Ny~\cite{dAClN25} provided a sufficient condition for extremality of Gibbs measures on any non-amenable, connected, transitive and locally finite graph.
In the case of hyperbolic lattices, they identified a family of extremal inhomogeneous Gibbs measures, indexed by a suitable collection of discrete geodesics of the dual lattice, under a sparsity condition on the disagreement edges.

The aim of the present paper is to extract uncountably many Series--Sinai states from the latter family, which gives a partial answer to Series--Sinai \cite{SS90}'s conjecture, described in section \ref{sec:results}.

\section{Results}
\label{sec:results}

Let \(p,\,q\ge{3}\) such that \(1/p+1/q<1/2\).
Let \(\Lpq\) be the tiling of the hyperbolic plane \(\Hyp{2}\) composed of regular \(p\)-gons such that each vertex of any \(p\)-gon has degree \(q\), see Figure \ref{fig:Lpq} for a representation of \(\Lpq\) in the Poincaré disk model for \((p,q) = (5,5)\) and \((p, q) = (6, 5)\).
By abuse of notation, we will also denote by \(\Lpq\) the underlying graph \((V_{p, q}, E_{p, q})\) given by the vertices of the \(p\)-gons and the edges between neighbouring vertices.
Let us denote by \(\ICpq\) the isoperimetric constant of \(\Lpq\) given by
\begin{equation}
    \ICpq
    =
    \inf_{
        \substack{
            \Lambda \Subset \Lpq \\
            \Lambda \neq \emptyset
        }
    }
    \frac{
        \cardinality{\partial_e \Lambda}
    }{
        \cardinality{\Lambda}
    },
\end{equation}
where
\(
    \partial_e\Lambda
    =
    \colonSet{
        e\in E
    }{
        \cardinality{
            e \cap \Lambda
        }
        =
        1
    }
\).
Häggstrom, Jonasson and Lyons computed in~\cite{HJL01} the value of \(\ICpq\) and remarked in~\cite[Remark 4.13]{HJL01} that the value was not attained by the balls of \(\Lpq\) given by the graph distance.
D'Achille, Jacquier and Ruszel proved in~\cite[Theorem 1.1]{dAJR25} that the bound was attained thanks to a layer decomposition due to Rietman--Nienhuis--Oitmaa and Moran.

\begin{figure}[ht]
    {
        \hfill
        \begin{subfigure}[l]{5cm}
            \centering
            \includegraphics[width=\textwidth]{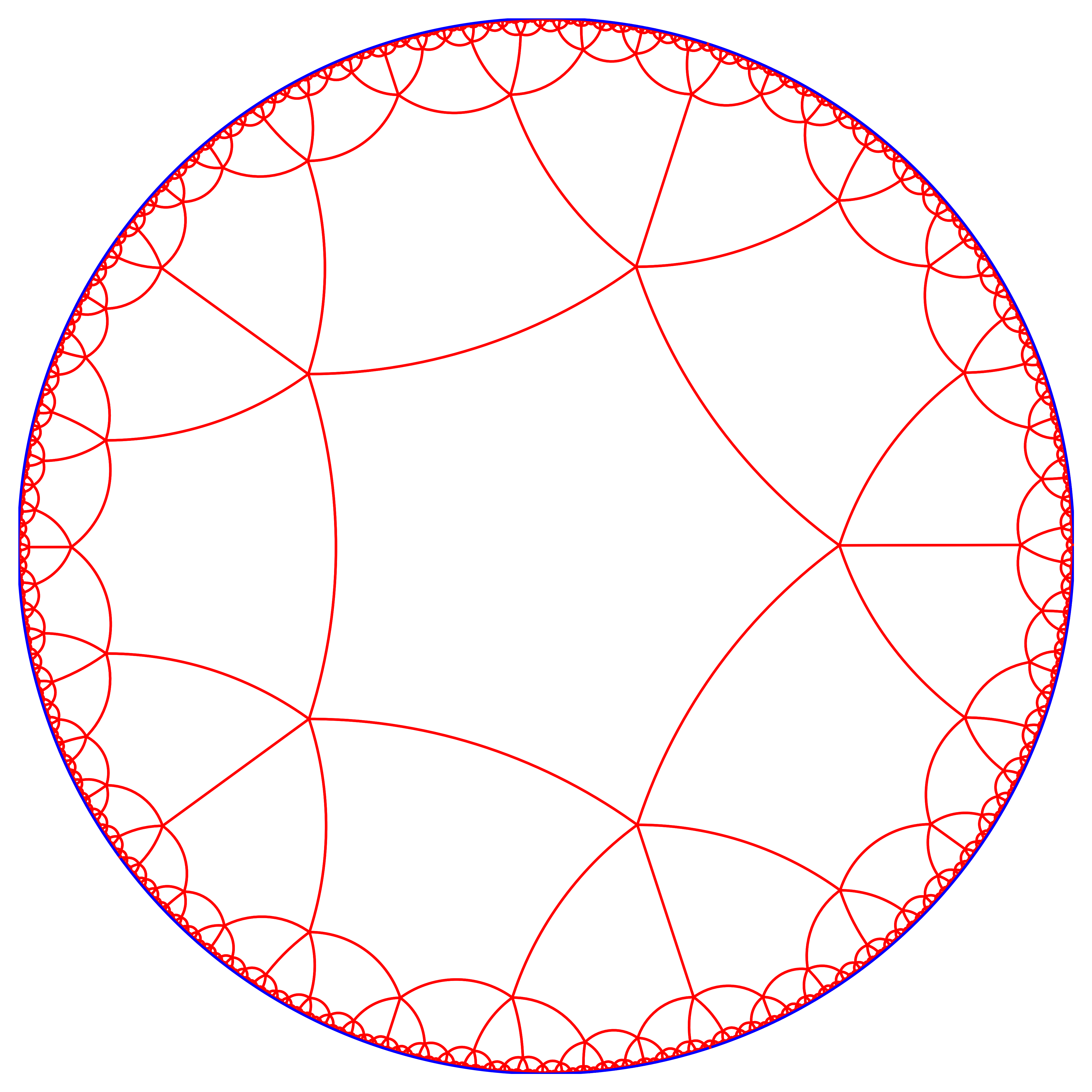}
            \subcaption{For \(p=5\) and \(q=5\).}
        \end{subfigure}
        \hfill
        \begin{subfigure}[r]{5cm}
            \centering
            \includegraphics[width=\textwidth]{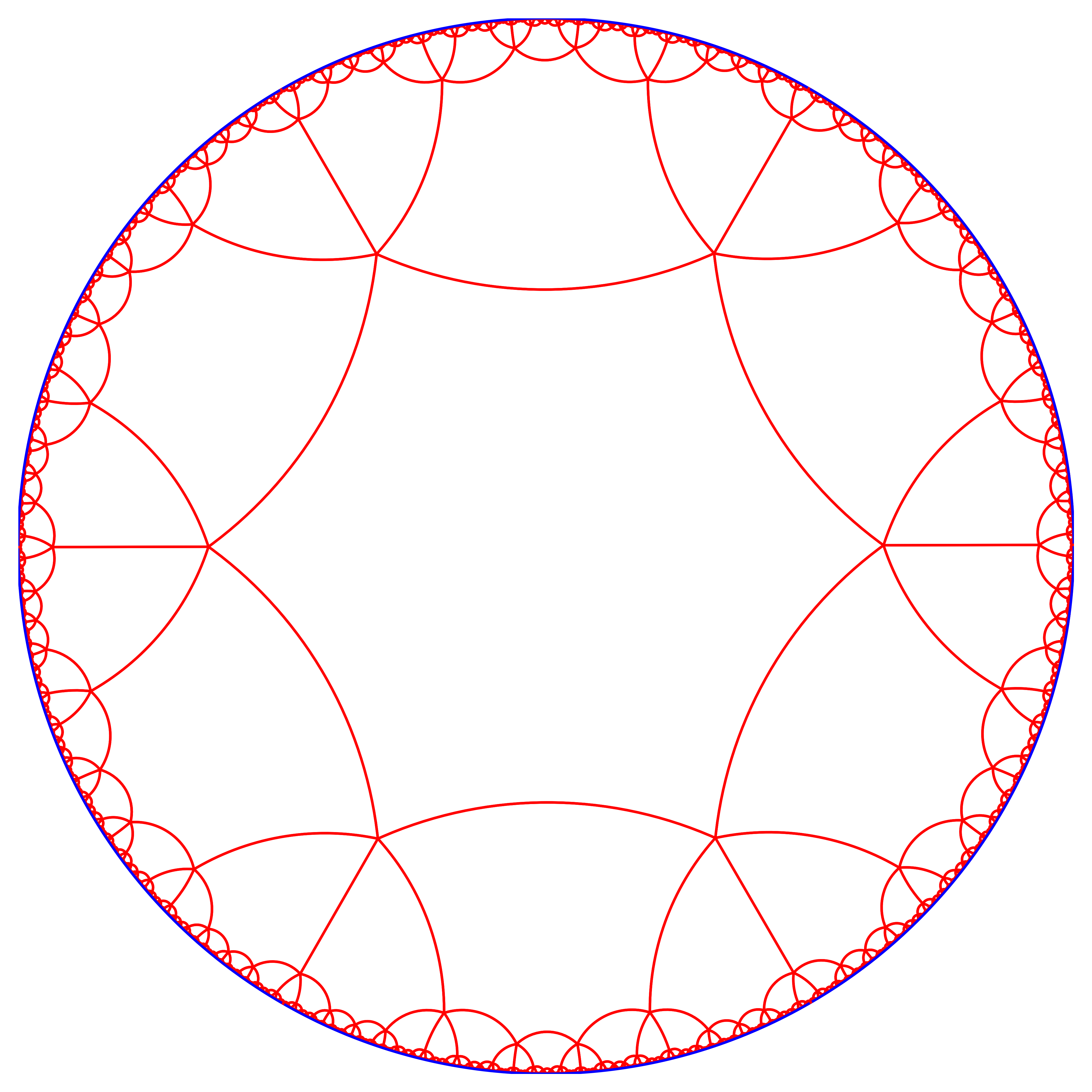}
            \subcaption{For \(p=6\) and \(q=5\).}
        \end{subfigure}
        \hfill
    }
    \caption{Two representations of \(\Lpq\) in the conformal disk model of \(\mathbb{H}_2\).}
    \label{fig:Lpq}
\end{figure}

Let \(G = (V, E)\) be a \emph{locally finite}, at most countable, \emph{connected} graph and let \(\Omega = \{-1, +1\}^V\) be the set of spin configurations.
For any \(\Lambda \Subset V\) finite subset of \(V\), we define
\begin{equation}
    H_\Lambda : \Omega\ni\omega \mapsto - \sum_{\{i, j\}\in E_\Lambda^e} \omega_i \omega_j \in\mathbb R,
\end{equation}
where \(E_\Lambda^e = \colonSet{e\in E}{e\cap\Lambda\neq\emptyset}\).
The Ising model on \(\Lambda\) at inverse temperature \(\beta\in(0, \infty)\) with boundary condition \(\eta\in\Omega\) is the probability measure \(\mu^\eta_{\Lambda;\beta}\) on the set of configurations
\begin{equation}
    \Omega_\Lambda^\eta
    =
    \colonSet{
        \omega\in\Omega
    }{
        \forall v\in V\setminus\Lambda,
        \,
        \omega_v = \eta_v
    }
\end{equation}
such that for all \(\omega\in\Omega^\eta_\Lambda\),
\begin{equation}
    \mu^\eta_{\Lambda;\beta}(\omega)
    =
    \frac{1}{Z^\eta_{\Lambda;\beta}}
    \exp\left(
        - \beta H_\Lambda(\omega)
    \right),
\end{equation}
where
\(
    \displaystyle
    Z^\eta_{\Lambda;\beta}
    =
    \sum_{\omega\in\Omega_\Lambda^\eta}
    \exp\left(
        - \beta H_\Lambda(\omega)
    \right)
\)
is the \emph{normalisation constant}, also called the \emph{partition function}.
The \emph{set of infinite volume Gibbs measures \(\mathcal G_\beta\) of the Ising model on \(G\) at inverse temperature \(\beta\)} is the set of all probability measures \(\mu\) on \(\Omega\) such that the \emph{DLR equations} hold, \ie~
\begin{equation}
\mathcal{G}_{\beta}
=
\colonSet{
    \mu \in \mathcal{M}_1(\Omega)
}{
    \forall \Lambda \Subset V,
    \,
    \mu
    =
    \int_\Omega \mu^\eta_{\Lambda;\beta} \, \ud \mu(\eta)
},
\end{equation}
where \(\mathcal{M}_1(\Omega)\) is the set of probability measures on \(\Omega\).
The set \(\mathcal G_\beta\) is a \emph{Choquet simplex}~\cite{Geo11}, and as such if we denote by \(\exG_\beta\) the set of its extremal elements, all Gibbs measures can be uniquely described as convex combinations of elements of \(\exG_\beta\).

Let \(f\) be an oriented simple curve of \(\Hyp{2}\) separating \(\Lpq\) into the set of vertices to the left of or on \(f\) and the set of vertices to the right of \(f\).
We denote by \(\eta^{f}\) the configuration such that for all \(v\in V_{p, q}\), \(\eta^{f}_v=+1\) if \(v\) is to the left of or on \(f\), and \(\eta^{f}_v=-1\) if \(v\) is to the right of \(f\).
One can wonder for which classes of \(f\) the weak limits \(\displaystyle\mu^{f}_{\beta} = \lim_{\Lambda\uparrow V} \mu^{\eta^{f}}_{\Lambda; \beta}\) exist, and what are their properties.

\begin{figure}[ht]
    {
        \hfill
        \begin{subfigure}[l]{5cm}
            \centering
            \begin{overpic}[
                percent=true,
                width=\textwidth
            ]{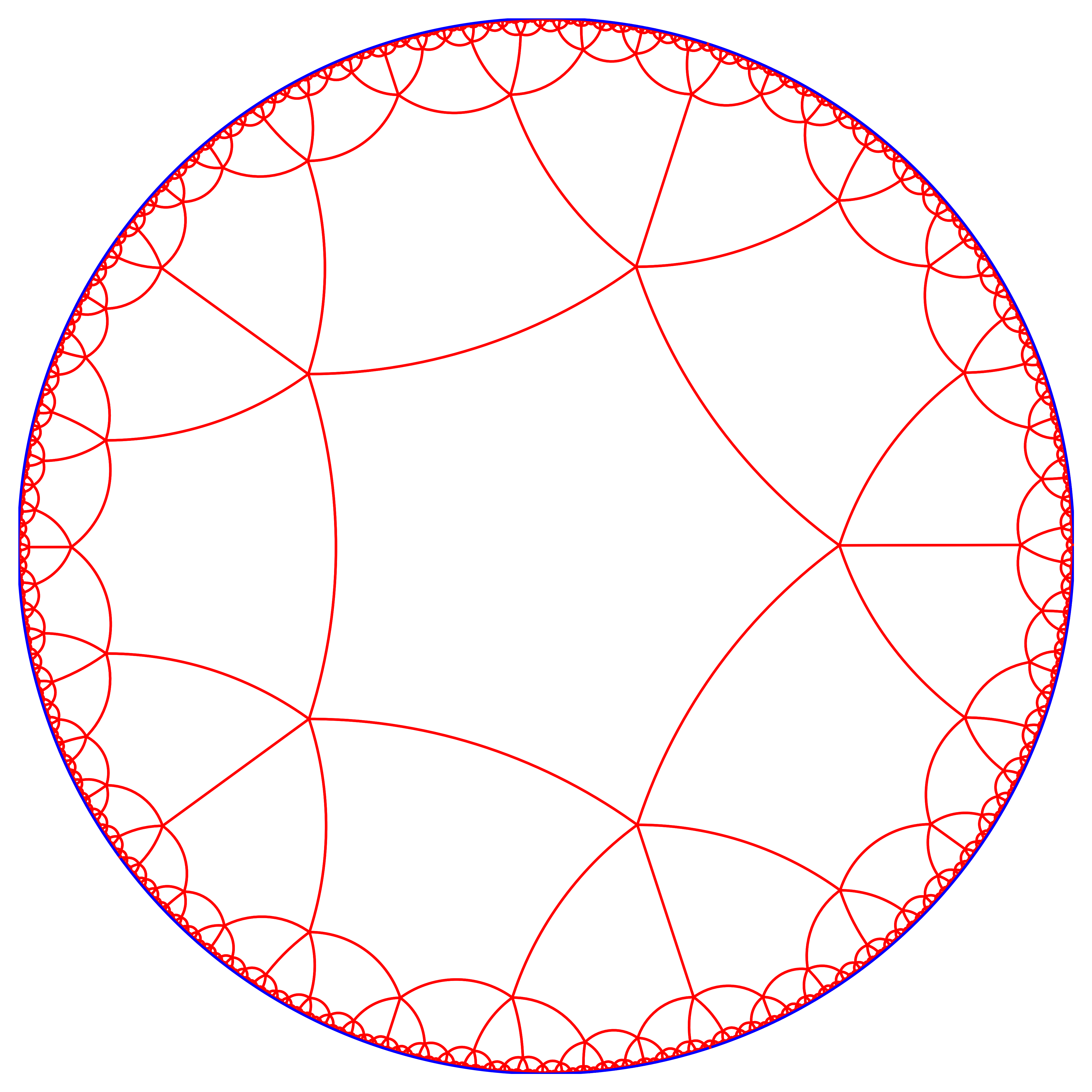}
                \put (85,85) {{\small \(+\)}}
                \put (95,73) {{\small \(+\)}}
                \put (100,59) {{\small \(+\)}}
                \put (100,45) {{\small \(+\)}}
                \put (97,31) {{\small \(+\)}}
                \put (88,17) {{\small \(+\)}}
                \put (79,8) {{\small \(+\)}}
                \put (71,2) {{\small \(+\)}}
                \put (57,-2) {{\small \(+\)}}
                \put (44,-2.5) {{\small \(+\)}}
                \put (18,5) {{\small \(+\)}}
                \put (5,17) {{\small \(+\)}}
                \put (0,30) {{\small \(+\)}}
                \put (31,0) {{\small \(+\)}}
                \put (-2,60) {{\small \(-\)}}
                \put (1,70) {{\small \(-\)}}
                \put (6,80) {{\small \(-\)}}
                \put (14,90) {{\small \(-\)}}
                \put (25,97) {{\small \(-\)}}
                \put (36,100) {{\small \(-\)}}
                \put (48,101) {{\small \(-\)}}
                \put (60,100) {{\small \(-\)}}

                \put (1,50){\tikz\draw[black,thick] (-1.5,0) arc (-90:-27.5:4);}
                \put (73,96) {\(\gamma\)}

                \put (49,49){\tiny\(\bullet\)}

                \put (52,50) {\(\origin\)}
            \end{overpic}
        \end{subfigure}
        \hfill
        \begin{subfigure}[r]{5cm}
            \centering
            \begin{overpic}[width=\textwidth]{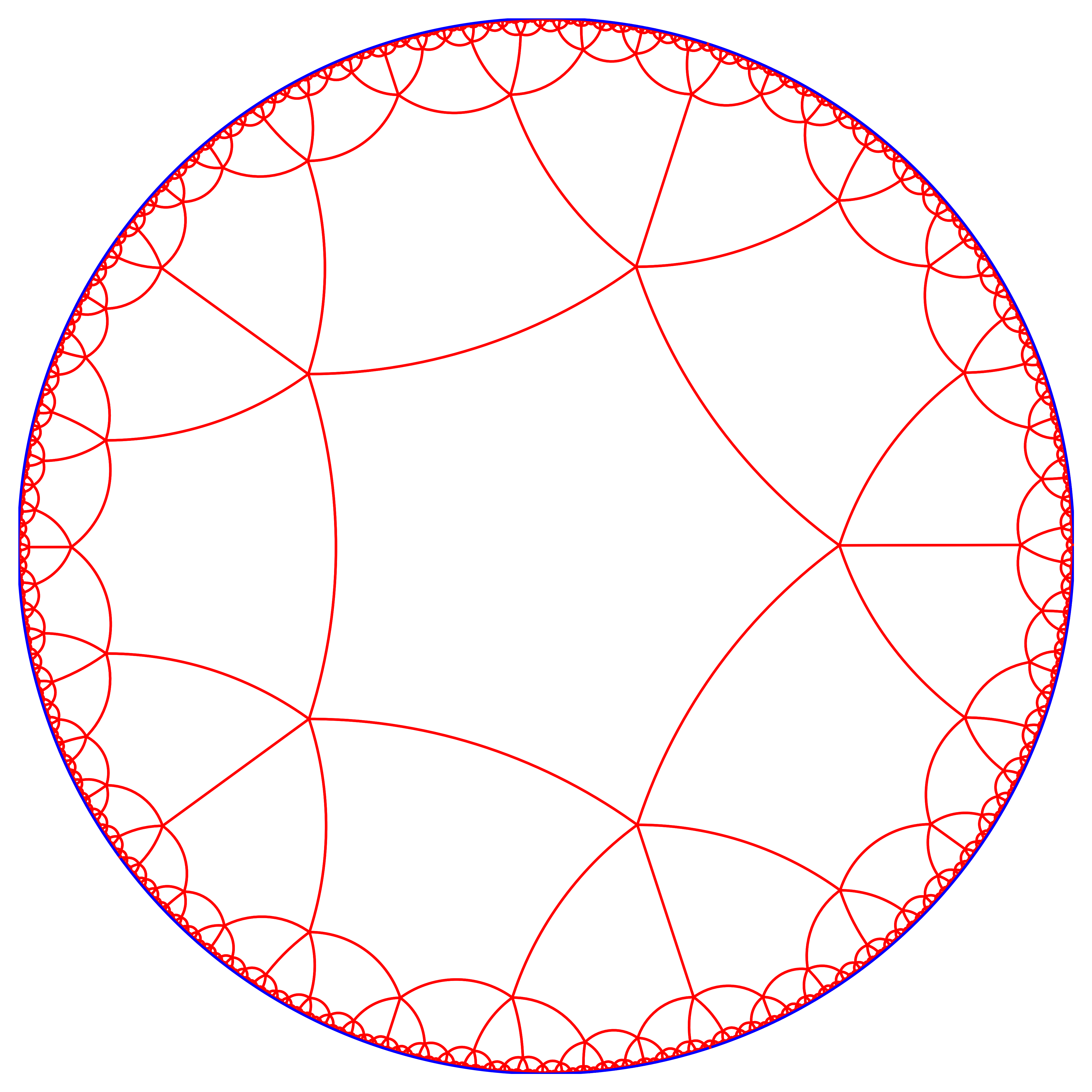}
                \put (65,98) {{\small \(-\)}}
                \put (75,93) {{\small \(-\)}}
                \put (95,73) {{\small \(+\)}}
                \put (100,59) {{\small \(+\)}}
                \put (100,45) {{\small \(+\)}}
                \put (97,31) {{\small \(+\)}}
                \put (88,17) {{\small \(+\)}}
                \put (79,8) {{\small \(+\)}}
                \put (71,2) {{\small \(+\)}}
                \put (57,-2) {{\small \(+\)}}
                \put (44,-2.5) {{\small \(+\)}}
                \put (18,5) {{\small \(-\)}}
                \put (5,17) {{\small \(-\)}}
                \put (0,30) {{\small \(-\)}}
                \put (31,0) {{\small \(-\)}}
                \put (-2,60) {{\small \(-\)}}
                \put (1,70) {{\small \(-\)}}
                \put (6,80) {{\small \(-\)}}
                \put (14,90) {{\small \(-\)}}
                \put (25,97) {{\small \(-\)}}
                \put (36,100) {{\small \(-\)}}
                \put (48,101) {{\small \(-\)}}
                \put (-3,44) {{\small \(-\)}}

                \put (85,85) {\color{darkgreen}\(\widehat{\gamma}\)}
                \put(0,0){
                    \color{darkgreen}
                    \polyline(40,3)(39,7)(38.5,22)(47,51)(72,67.5)(80,80)(81.5,83.5)
                }
                \put (-1.3,-1){
                    \color{darkgreen}
                    \put (39,7){\tiny\(\bullet\)}
                    \put (38.5,22){\tiny\(\bullet\)}
                    \put (72,67.5){\tiny\(\bullet\)}
                    \put (80,80){\tiny\(\bullet\)}
                }

                \put (52,50) {\(\origin\)}
                \put (48.4,49.9){\tiny\(\bullet\)}
            \end{overpic}
        \end{subfigure}
        \hfill
    }
    \caption{
        Left: a Series--Sinai state for the geodesic \(\gamma\) on \(\fLpq{5}{5}\).
        Right: a zigzag state for the path \(\widehat\gamma\) on \(\fLpq{5}{5}\).
    }
    \label{fig:states}
\end{figure}

Series and Sinai~\cite{SS90} studied the class \(\Gamma\) of geodesics of \(\Hyp{2}\).
They proved that for \(p,\,q\ge{3}\) such that \(1/p+1/q<1/2\) and for any geodesic \(\gamma\in\Gamma\), if we consider \(\mu^{\gamma}_{\beta}\) a possibly subsequential weak limit of finite volume states with boundary condition \(\eta^\gamma\), then for every \(\gamma_1\) and \(\gamma_2\) distinct elements of \(\Gamma\), the measures \(\mu_\beta^{\gamma_1}\) and \(\mu_\beta^{\gamma_2}\) are mutually singular at low temperature, which gives that \(\exG_\beta\) is uncountable at low temperature.
We will call \emph{Series--Sinai states} the following family of probability measures
\begin{equation}
    \mathrm{SS}_\beta(p, q)
    =
    \commaSet{
        \mu_\beta^{\gamma}
    }{
        \gamma \in \Gamma
    }.
\end{equation}
They stated the following conjecture.

\begin{conjecture}
[Extremality of all Series--Sinai states, \cite{SS90}]
\label{conj:ss_conjecture}
    Let \(p,\,q\ge{3}\) such that \(1/p+1/q<1/2\).
    For \(\beta \in (0, \infty)\) large enough,
    \begin{equation}
        \mathrm{SS}_{\beta}(p,q)
        \subset
        \exG_\beta.
    \end{equation}
\end{conjecture}

For \(p \ge{5}\) and \(q \ge{3}\) with \(1/p+1/q<1/2\), D'Achille, Coquille and Le~Ny~\cite{dAClN25} studied the set \(\widehat\Gamma(p,q)\) of complete paths \(\widehat\gamma\) in the dual lattice with the property that whenever \(\widehat\gamma\) crosses an edge of \(\Lpq\), it cannot cross the adjacent edges of the lattice see Figure \ref{fig:states} for an example and paragraph \ref{par:hyp_til} for a formal definition.
We will call \emph{zigzag paths} the paths of \(\widehat\Gamma(p,q)\).
Thanks to an \emph{excess energy estimate}, they proved that \(\mu_\beta^{\widehat\gamma}\) is well-defined and extremal at low temperature.
We will call \emph{zigzag states} the following family of probability measures
\begin{equation}
    \mathrm{Zigzag}_\beta(p, q)
    =
    \commaSet{
        \mu_\beta^{\widehat\gamma}
    }{
        \widehat\gamma \in \widehat\Gamma(p,q)
    }.
\end{equation}

The two classes \(\mathrm{SS}_{\beta}(p,q)\) and \(\mathrm{Zigzag}_\beta(p, q)\) can be linked to one another, through the following lemma.

% \nopagebreak[5]
\begin{lemma}[Extremality of Series--Sinai states via dual lattice proximity]
\label{lem:link_SS_DACLN}
    Let \(p,\,q\ge{3}\) with \(1/p+1/q<1/2\).
    Let \(\gamma\in\Gamma\), if there exists \(\widehat\gamma\in\widehat\Gamma(p, q)\) such that the Hausdorff distance between \(\gamma\) and \(\widehat\gamma\) is smaller than half the side length of a \(p\)-gon of \(\Lpq\), then the configurations associated to \(\gamma\) and \(\widehat\gamma\) are the same and \(\mu^\gamma_\beta=\mu^{\widehat{\gamma}}_\beta\).
    In particular, the Series--Sinai state associated to \(\gamma\) is an extremal state at low enough temperature.
\end{lemma}

\begin{figure}[ht]
    \centering
    \begin{overpic}[width=5cm]{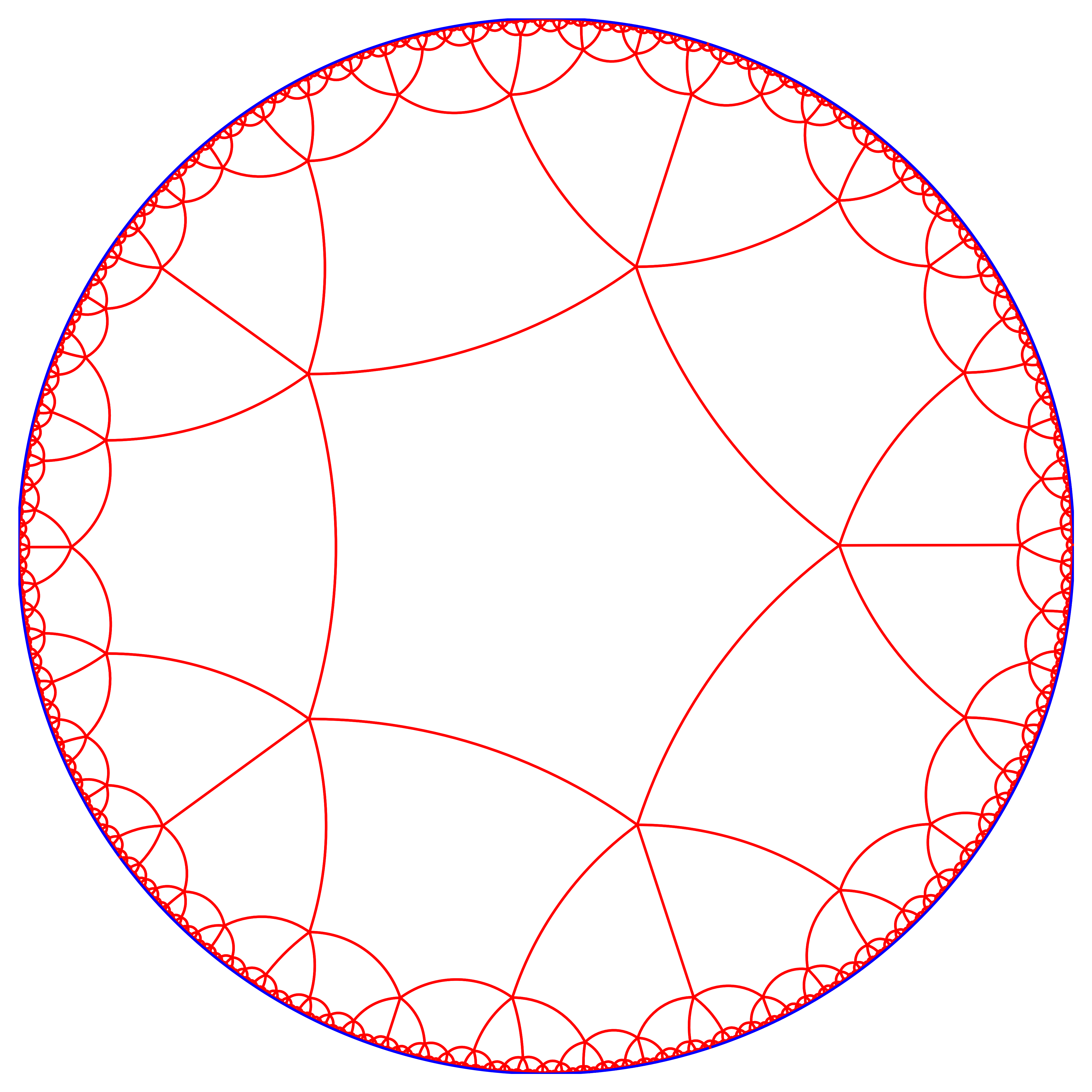}
        \put (65,98) {{\small \(-\)}}
        \put (75,93) {{\small \(-\)}}
        \put (95,73) {{\small \(+\)}}
        \put (100,59) {{\small \(+\)}}
        \put (100,45) {{\small \(+\)}}
        \put (97,31) {{\small \(+\)}}
        \put (88,17) {{\small \(+\)}}
        \put (79,8) {{\small \(+\)}}
        \put (71,2) {{\small \(+\)}}
        \put (57,-2) {{\small \(+\)}}
        \put (18,5) {{\small \(-\)}}
        \put (5,17) {{\small \(-\)}}
        \put (0,30) {{\small \(-\)}}
        \put (31,0) {{\small \(-\)}}
        \put (-2,60) {{\small \(-\)}}
        \put (1,70) {{\small \(-\)}}
        \put (6,80) {{\small \(-\)}}
        \put (14,90) {{\small \(-\)}}
        \put (25,97) {{\small \(-\)}}
        \put (36,100) {{\small \(-\)}}
        \put (48,101) {{\small \(-\)}}
        \put (-3,44) {{\small \(-\)}}

        \put (85,85) {\color{darkgreen}\(\widehat{\gamma}\)}
        \put(0,0){
            \color{darkgreen}
            \polyline(40,3)(39,7)(38.5,22)(47,51)(72,67.5)(80,80)(81.5,83.5)
        }
        \put (-1.3,-1){
            \color{darkgreen}
            \put (39,7){\tiny\(\bullet\)}
            \put (38.5,22){\tiny\(\bullet\)}
            \put (72,67.5){\tiny\(\bullet\)}
            \put (80,80){\tiny\(\bullet\)}
        }
        
        \put (36.5,1){\tikz\draw[black,thick] (0,0) arc (190:115:3.85);}
        \put (40,-3) {\(\gamma\)}

        \put (52,50) {\(\origin\)}
        \put (48.4,49.9){\tiny\(\bullet\)}
    \end{overpic}
    \caption{Illustration of Lemma \ref{lem:link_SS_DACLN}: when a geodesic is at small enough distance of a zigzag path, they both define the same partition of \(V_{p,q}\).}
\end{figure}

Motivated by Lemma \ref{lem:link_SS_DACLN}, the main result we obtain is the following.

\begin{theorem}[Extremal Series--Sinai states]
\label{thm.principal}
For all \(p\ge{5}\), there exists \(q_p\ge{3}\) such that for all \(q \ge{q_p}\), there exist uncountably many \emph{extremal} Series--Sinai states for the ferromagnetic nearest-neighbour Ising model on \(\Lpq\) at low enough temperature.
\end{theorem}

% \paragraph*{Goal of the paper.}

To prove Theorem \ref{thm.principal}, we compute a bound on the remoteness between a selected path of \(\widehat{\Gamma}(p, q)\) and the geodesic path joining its endpoints.
More precisely, we consider \(\widehat{\gamma}\in\widehat{\Gamma}(p, q)\) such that the angles between two consecutive geodesic segments is small enough, and bound the Hausdorff distance between \(\widehat{\gamma}\) and the geodesic \(\gamma\) joining its endpoints.

\section{Tools}

\paragraph{Hausdorff distance, quasi-geodesics and \(\delta\)-hyperbolic metric spaces.}

Let \((X,\distance{}{})\) be a \emph{proper geodesic} metric space.
Let \(A\) and \(B\) be two non-empty subsets of \(X\).
The \emph{Hausdorff distance \(\hausdorffDistance{A}{B}\) between \(A\) and \(B\)} is defined by
\begin{equation}
    \hausdorffDistance{A}{B}
    =
    \max\left(
        \sup_{a\in A} \distance{a}{B},\;
        \sup_{b\in B} \distance{b}{A}
    \right)
\end{equation}
where \(\displaystyle\distance{a}{B} = \inf_{b\in B} \distance{a}{b}\) is the \emph{distance of the point \(a\) to the set \(B\)}.
For any pair of points \((x,y) \in X^2\), we denote by \([x,y]\) the \emph{geodesic segment} joining \(x\) and \(y\).
For every finite sequence of points \(x_1,\,\ldots,\,x_n \in X\), we will denote by \([x_1,\ldots,x_n]\) or by \([x]\) the union of all the geodesic segments \([x_i,x_{i+1}]\).
Given this notation, given three points \(x,\,y,\,z\in X\), the \emph{geodesic triangle} \(\triangle(x,y,z)\) is the union of the three geodesics \([x,\,y]\), \([y,z]\) and \([z,x]\), \ie~\(\triangle(x, y, z) = [x, y, z, x]\).

\begin{definition}[\(\delta\)-hyperbolic metric spaces,~\cite{BH13}]
    Let \(\delta > 0\).
    A geodesic triangle in a metric space is said to be \emph{\(\delta\)-thin} \textiff{} each of its sides is contained in the \(\delta\)-neighbourhood of the union of the other two sides, \ie~every point on a side of the triangle is at distance at most \(\delta\) from the union of the other two sides.
    A geodesic metric space \(X\) is said to be \emph{\(\delta\)-hyperbolic} \textiff{} every geodesic triangle in \(X\) is \(\delta\)-thin.
    Moreover, if \(X\) is \(\delta\)-hyperbolic for some \(\delta > 0\), then \(X\) is said to be \emph{hyperbolic}.
\end{definition}

\begin{definition}[Quasi-geodesics]
    Let \(k \ge 1\) and \(c \ge 0\) and let \(I\) be a \emph{closed} interval of \(\mathbb{R}\).
    A \emph{\((k, c)\)-quasigeodesic curve} \(f\) of \(X\) is an application \(f:I\to X\) such that for all \(s,\,t \in I\), we have
    \begin{equation}
        \frac{1}{k} \abs{t-s} - c
        \le 
        \distance{f(s)}{f(t)}
        \le 
        k \abs{t-s} + c.
    \end{equation}
    Moreover, if \(I\) is \emph{bounded} (\resp~\(I = \mathbb{R}\)), we say that \(f\) is a \emph{\((k, c)\)-quasigeodesic segment} (\resp~\emph{complete \((k, c)\)-quasi-geodesic}).
\end{definition}

The main point behind the notion of quasi-geodesic is that any quasigeodesic curve in a hyperbolic space is close to a geodesic curve, \ie~we can bound the Hausdorff distance between them by a constant that only depends on the constants of quasi-geodicity of the curve and not on its length.
In particular, the following lemma holds for all \((k, c)\)-quasigeodesic curves, and especially for complete \((k, c)\)-quasi-geodesics.

\begin{lemma}[Morse--Mostow lemma]
\label{lem:Morse--Mostow}
    Let us assume that \(X\) is \emph{\(\delta\)-hyperbolic}.
    Then, there exists a constant \(C > 0\) such that for all \(k\ge 1\) and \(c\ge 0\) and all \((k, c)\)-quasigeodesic segment \(f:[0, t]\to \Hyp{2}\), we have
    \begin{equation}
        \hausdorffDistance{f([0, t])}{[f(0), f(t)]} \le C k^2 (c + \delta).
    \end{equation}
\end{lemma}

Note that in the previous lemma, the constant \(C\) is not explicited, however to the best of our knowledge, the best constant is \(C = 92\) \cite{GS19}.

\paragraph{The Poincaré disk of curvature \(-1\).}

Let us denote by \(\mathbb D = \colonSet{z\in \mathbb{C}}{|z| < 1}\) the \emph{Poincaré disk\footnote{This provides an occurrence of Stigler's law of eponymy (or even of the first Arnold principle~\cite{Arnold}), since the conformal disk model of \(\Hyp{2}\) was introduced by Beltrami in 1868, see~\cite{arcozzi}.} of curvature \(-1\)} equipped with the metric \(\hyperbolicDistance{}{}\) such that for all \(u,\,v\in \mathbb D\), we have
\begin{equation}
    \hyperbolicDistance{u}{v}
    = \log
        \frac{
            \abs{1 - u \overline v} + \abs{u - v}
        }{
            \abs{1 - u \overline v} - \abs{u - v}
        }
    = 2 \artanh \abs{\frac{u-v}{1-u\overline v}},
\end{equation}
as \(\displaystyle \artanh : (-1, 1) \ni x \mapsto \frac{1}{2} \log\frac{1+x}{1-x} \in \mathbb{R}\).
We have in particular for any \(u\in\mathbb{D}\),
\begin{equation}
    \hyperbolicDistance{0}{u}
    =
    \log \frac{1 + \abs{u}}{1 - \abs{u}}
    =
    2 \artanh \abs{u}.
\end{equation}
\(\mathbb D\) is a proper geodesic space and for \(u,\,v\in\mathbb D\), \([u, v]\) is the arc of the circle which is orthogonal to the unit circle passing through \(u\) and \(v\).
In this setting, a curve \([x_1,\ldots,x_n]\) is uniquely determined, up to a hyperbolic rotation around \(x_1\), by \(x_1\), the lengths \(\ell_i = \hyperbolicDistance{x_i}{x_{i+1}}\) and the angles \(\theta_i = \angle(x_{i-1},x_i,x_{i+1})\) oriented in the counter-clockwise direction between \([x_{i-1},x_i]\) and \([x_i,x_{i+1}]\).
Let us denote by \(\widehat\gamma(x_1, \theta, \ell)\) such a path.
\(\mathbb D\) is a \emph{\(\delta\)-hyperbolic metric space} for \(\delta = \log\left(1 + \sqrt 2\right)\) (see~\cite{And05} Exercice 5.11).

\paragraph{Hyperbolic tilings \(\Lpq\) of \(\mathbb{D}\)~\cite{GS87}.}
\label{par:hyp_til}

Let \(p,\,q \ge{ 3}\) such that \(1/p+1/q < 1/2\).
We define the hyperbolic tiling \(\Lpq = (V_{p, q}, E_{p, q}, F_{p, q})\), where \(F_{p,q}\) is the set of all faces of the tiling, as the tiling of the Poincaré disk \(\mathbb{D}\) with regular \(p\)-gons of side length \(\ell_{p, q}\) such that each vertex of the \(p\)-gons has degree \(q\).
The dual graph\footnote{The dual graph of a planar graph \(G=(V,E)\) is the graph whose vertices are faces of \(G\) and two faces are neighbours \textiff{} they share a common edge.} of \(\Lpq\) is isomorphic to \(\mathcal{L}_{q, p}\).

\begin{wrapfigure}[11]{r}{.6\textwidth}
    \centering
    \begin{overpic}[
        height=3.5cm,
    ]{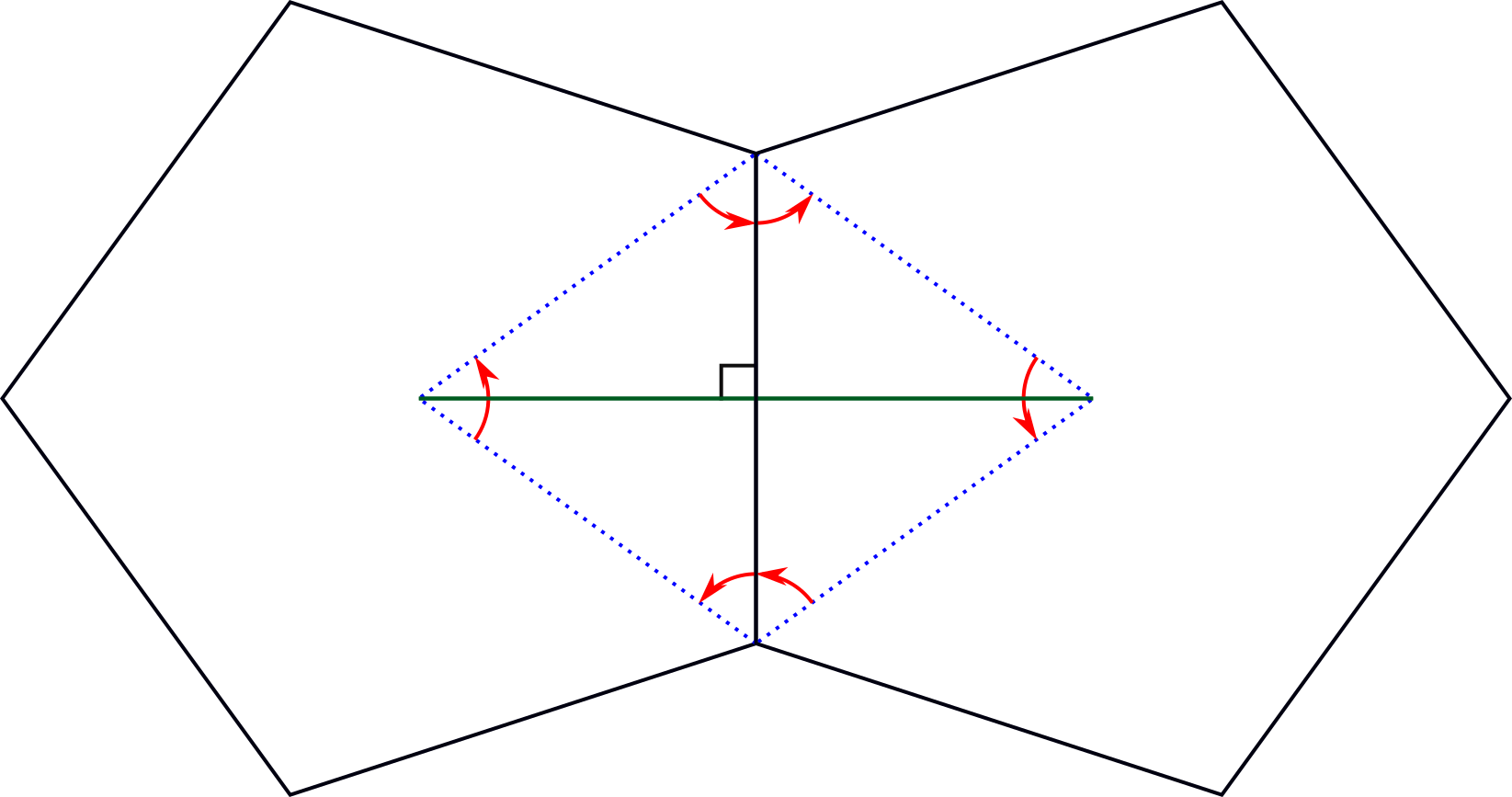}
        \put (26.5,25.5) {\tiny\(\bullet\)}
        \put (71.5,25.5) {\tiny\(\bullet\)}
        
        \put (44.5,34.5) {\color{red} \(\frac{\pi}{q}\)}
        \put (44.5,17.5) {\color{red} \(\frac{\pi}{q}\)}
        
        \put (33,25) {\color{red} \(\frac{2\pi}{p}\)}

        \put(26.8,25.5){\tikz \draw[dashed](0,0)--(0,2);}
        \put(26.8,55){\tikz \draw[<->](0,0)--(3,0);}
        \put(71.8,25.5){\tikz \draw[dashed](0,0)--(0,2);}
        \put (45,58) {\(\ell_{q,p}\)}

        \put (50,10) {\tikz \draw[dashed] (0,0)--(3.5,0);}
        \put (103,10){\tikz \draw[<->](0,0)--(0,2.1);}
        \put (50,42.3) {\tikz \draw[dashed] (0,0)--(3.5,0);}
        \put (105,26) {\(\ell_{p,q}\)}
    \end{overpic}
    \caption{Representation of part of a cell of \(\Lpq\) with \(p=5\).}
    \label{fig:cell}
\end{wrapfigure}
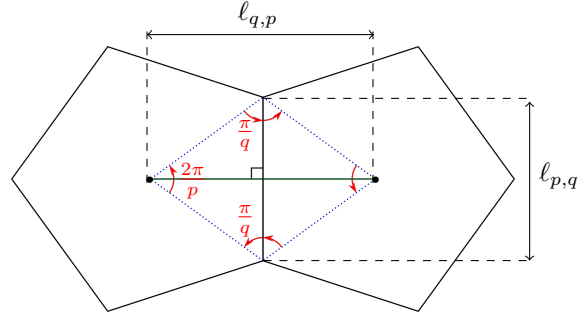

The length \(\ell_{p, q}\) is computed in~\cite{Coxeter}
\begin{equation}
    \ell_{p, q}
    =
    2 \arcosh\left(
        \frac{
            \cos(\pi/p)
        }{
            \sin(\pi/q)
        }
    \right),
\end{equation}
as can be seen by considering a triangle inside a cell of \(\Lpq\) as depicted in Figure \ref{fig:cell}.

Let \((f_i)_{i\in\mathbb{Z}}\in{F_{p,q}}^\mathbb{Z}\) be a collection of faces of \(\Lpq\) such that for all \(i\in\mathbb{Z}\), \(f_i\) and \(f_{i+1}\) share an edge but \(f_{i-1}\) and \(f_{i+1}\) have no common vertices.
Let \((u_i)_{i\in\mathbb{Z}}\) be the sequence of points of \(\Hyp{2}\) such that for \(i\in\mathbb{Z}\), \(u_i\) is the center of \(f_i\).
We call \emph{zigzag path} the path \(\displaystyle\widehat{\gamma}=\bigcup_{i\in\mathbb Z}[u_i,u_{i+1}]\) and we denote by \(\widehat{\Gamma}(p,q)\) the set of all zigzag paths.

\paragraph{Discrete hyperbolic Schur theorem.}

Schur's theorem allows to bound from below the distance between two endpoints of a continuous path by a constant depending only on its curvature.
We are interested in a discrete version of it.

\begin{theorem}
[Discrete hyperbolic Schur theorem~\cite{FG97}]
\label{lem:schur}
    For \(\ell > 0\), set 
    \begin{equation}
        \alpha_\ell = 2 \arctan\left(\sinh\left(\frac{\ell}{2}\right)\right).
    \end{equation}
    Let us consider \(u_0,\,\ldots,\,u_n\) and \(v_0,\,\ldots,\,v_n\) points of \(\Hyp{2}\).
    Let us suppose that for all \(i\in\llbracket 0, n-1\rrbracket\), \(\hyperbolicDistance{u_i}{u_{i+1}} = \hyperbolicDistance{v_i}{v_{i+1}} = \ell\).
    Let us suppose that for all \(i\in\llbracket 1, n-1\rrbracket\), if we denote \(\theta_i = \angle(u_{i-1}, u_i, u_{i+1})\in[-\pi, \pi]\), then \(\angle(v_{i-1}, v_i, v_{i+1}) = \abs{\theta_i}\).
    If for all \(i\in\llbracket 1, n-1\rrbracket\), \(\abs{\theta_i} \ge \pi - \alpha_\ell\), we have
    \begin{equation}
        \hyperbolicDistance{v_0}{v_n}
        \le
        \hyperbolicDistance{u_0}{u_n}.
    \end{equation}
\end{theorem}

\begin{figure}[ht]
    \centering
    \begin{overpic}[
        width=5cm
    ]{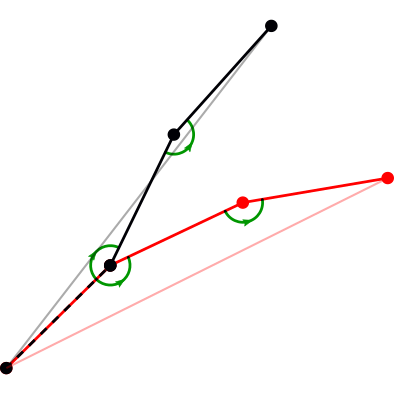}
        \put (-8, 6) {\color{black} \(u_0\)}
        \put (24, 41) {\color{black} \(u_1\)}
        \put (35, 65) {\color{black} \(u_2\)}
        \put (65, 97) {\color{black} \(u_3\)}

        \put (0, 0) {\color{red} \(v_0\)}
        \put (35, 30) {\color{red} \(v_1\)}
        \put (59, 53) {\color{red} \(v_2\)}
        \put (95, 48) {\color{red} \(v_3\)}

        \put (16, 35) {\color{darkgreen} \(\theta_1\)}
        \put (27, 20) {\color{darkgreen} \(\abs{\theta_1}\)}
        \put (50, 60) {\color{darkgreen} \(\theta_2\)}
        \put (60, 35) {\color{darkgreen} \(\abs{\theta_2}\)}
    \end{overpic}
    \caption{Illustration of Theorem \ref{lem:schur}: when the angles are large enough, the points \(v_0\) and \(v_3\) are closer from one another than the points \(u_0\) and \(u_3\).}
\end{figure}
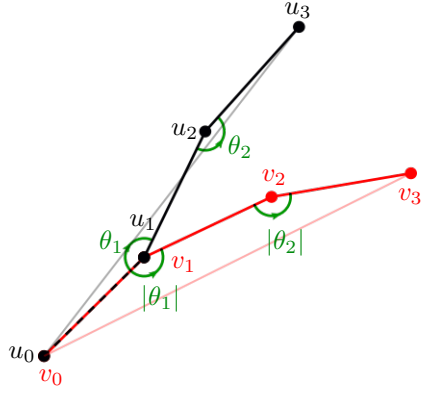

The angle \(\alpha_\ell\) corresponds to the angle such that the points defining \(\widehat{\gamma}(w_0, \pi-\alpha_\ell, \ell)\) all lie on the same horocycle, and, as such, separates the case where choosing paths at distance \(\ell\) from one another will create a self-intersecting path and the case where it will not.
In the case where \(\ell=\ell_{q, p}\), we have
\begin{equation}
    \label{eq:alphaqp_inequality}
    \alpha_{\ell_{q, p}}
    =
    2
    \arctan
    \left(
        \sqrt{
            \frac{
                \cos^2(\pi/q)
            }{
                \sin^2(\pi/p)
            }
            -
            1
        }
    \right)
    >
    \frac{
        2 \pi
    }{
        p
    },
\end{equation}
as soon as \(\tan(\pi/p) < \cos(\pi/q)\), see Figure \ref{fig:validity_domain} to see the points where this inequality is verified.
Compared to the condition \(1/p+1/q<1/2\), this condition changes the case when \(p\le 5\): for \(p\le4\), the condition cannot be fulfilled, and for \(p=5\), the condition becomes \(q\ge5\) rather than \(q\ge4\).

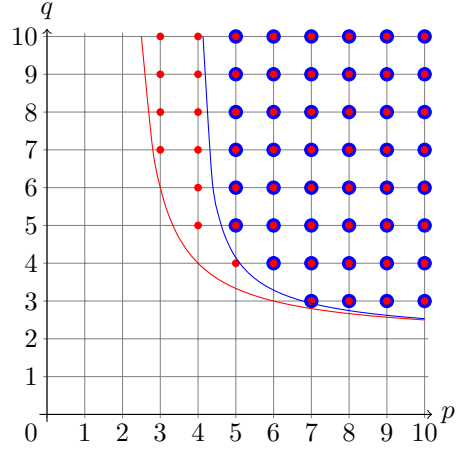
\begin{figure}[ht]
    \centering
    \begin{tikzpicture}[scale=0.5]
        \draw[very thin,color=gray] (-0.1,-0.1) grid (10.1,10.1);

        \draw[->] (-0.2,0) -- (10.2,0) node[right] {\(p\)};
        \draw[->] (0,-0.2) -- (0,10.2) node[above] {\(q\)};

        \draw (0,0) node[below left]{0};
        \foreach \l in {1,...,10}{
            \draw (0,\l) node[left]{\l};
            \draw (\l,0) node[below]{\l};
        }

        \foreach \p in {3,...,10}{
            \foreach \q in {3,...,10}{
                \pgfmathparse{tan(180/\p)<cos(180/\q)?1:0}
                \ifnum\pgfmathresult>0 
                    \filldraw[blue] (\p,\q) circle (5pt);
                \fi

                \pgfmathparse{int((\p-2)*(\q-2))}
                \ifnum\pgfmathresult>4
                    \filldraw[red] (\p,\q) circle (2.5pt);
                \fi
            }
        }

        \draw[color=red,domain=2.5:10,smooth] plot (\x,{2*\x/(\x-2)});
        \draw[color=blue,domain=4.132:10,smooth] plot (\x,{180/acos(tan(180/\x))});
    \end{tikzpicture}
    \caption{The red points verify \(1/p+1/q<1/2\), the blue ones verify \(\tan(\pi/p)<\cos(\pi/q)\).}
    \label{fig:validity_domain}
\end{figure}

\section{Proofs}

Let us set for this section \(p \ge{5}\) and \(q\ge{3}\) such that \(1/p + 1/q < 1/2\).
Let us set \(a_p\) by
\begin{equation}
a_p
=
\begin{cases}
    \frac{\pi}{p} & \text{, if \(p\) is odd;} \\
    \frac{2\pi}{p} & \text{, if \(p\) is even.}
\end{cases}
\end{equation}
These angles are represented on Figure \ref{fig:a_p}.
We consider the zigzag paths \(\widehat{\gamma}\) such that two consecutive geodesic parts of \(\widehat{\gamma}\) have angles \(\pm (\pi-a_p)\) between them.
Moreover, for such a path, when we denote \(\displaystyle\widehat\gamma = \bigcup_{i\in\mathbb{Z}} [u_i, u_{i+1}]\), then we take the sequence \((u_i)_{i\in\mathbb{Z}}\) defining the zigzag paths in paragraph \ref{par:hyp_til}.
Let us denote by \(\widehat{\Gamma}_\mathrm{red}(p,q)\subset\widehat{\Gamma}(p,q)\) the set of all these zigzag paths \ie
\begin{equation*}
    \widehat{\Gamma}_\mathrm{red}(p,q)
    =
    \colonSet{
        \widehat{\gamma}(u_0,\theta,\ell_{q,p})
        =
        \bigcup_{i\in\mathbb{Z}}
        [u_i, u_{i+1}]
    }{
        \forall i \in \mathbb{Z},
        \,
        \theta_i \in \left\{\pm A_p\right\}
        \text{ and } u_i
        \text{ is the center of a face of \(\Lpq\)}   
    }.
\end{equation*}

\begin{figure}[ht]
    {
        \hfill
        \begin{subfigure}[l]{5cm}
            \centering
            \begin{overpic}[
                width=5cm
            ]{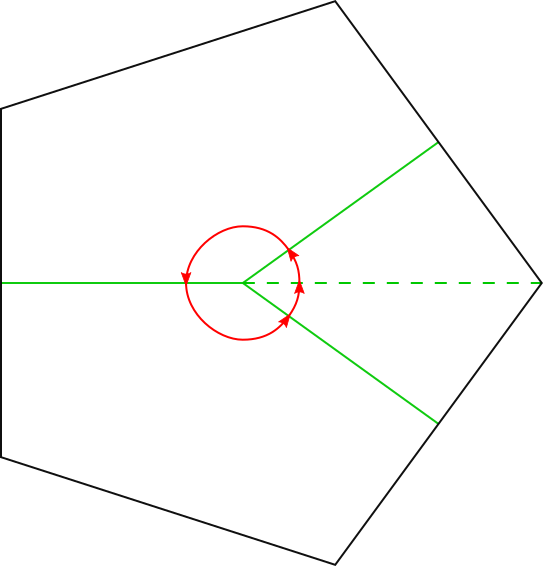}
            \put (28, 60) {\color{red} \(A_5\)}
            \put (28, 37) {\color{red} \(A_5\)}
            \put (55, 44) {\color{red} \(a_5\)}
            \put (55, 54) {\color{red} \(a_5\)}
            \end{overpic}
            \subcaption{\(p=5\)}
        \end{subfigure}
        \hfill
        \begin{subfigure}[r]{5cm}
            \centering
            \begin{overpic}[
                width=5cm
            ]{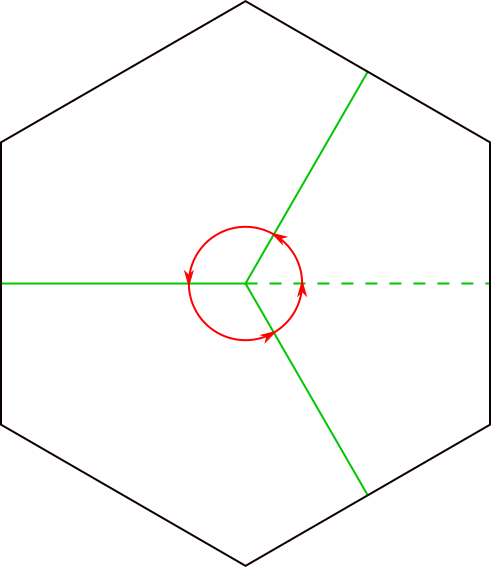}
            \put (28, 60) {\color{red} \(A_6\)}
            \put (28, 37) {\color{red} \(A_6\)}
            \put (55, 44) {\color{red} \(a_6\)}
            \put (55, 54) {\color{red} \(a_6\)}
            \end{overpic}
            \subcaption{\(p=6\)}
        \end{subfigure}
        \hfill
    }
    \caption{Motivation behind the choice of the angles: we want to make sure to extract a sub-tree of the dual lattice that goes as forward as possible.}
    \label{fig:a_p}
\end{figure}
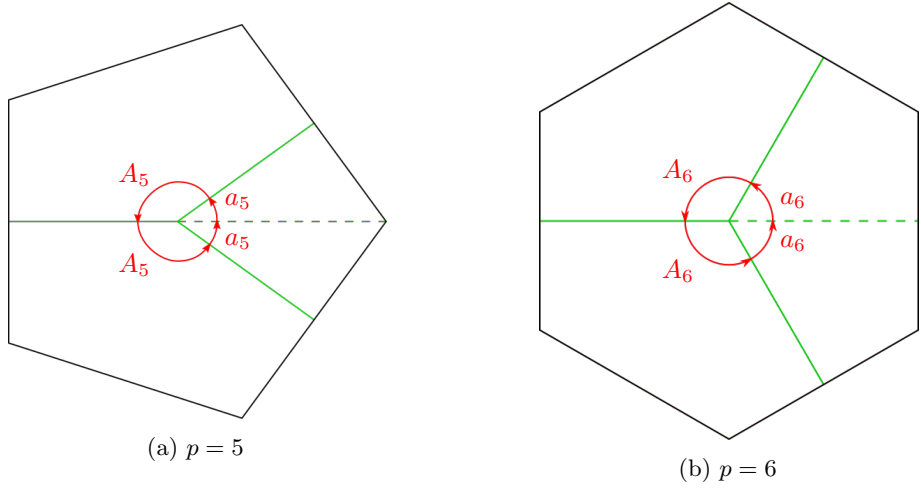

\paragraph*{Organisation of the proofs.}

In order to prove Theorem \ref{thm.principal}, we prove in Subsection \ref{ssec:cons_q_geod} that the previous \(\widehat{\gamma}\) are quasi-geodesics, then use the Morse--Mostow lemma in Subsection \ref{ssec:haus_dist_esti} to have a bound on the distance between \(\widehat{\gamma}\) and the geodesic \(\gamma\) at finite distance from it.
Finally, in Subsection \ref{ssec:proof_main_theorem} we find conditions on \(p\) and \(q\) such that this distance is small enough.

\subsection{Constants of quasi-geodicity}
\label{ssec:cons_q_geod}

The goal of this subsection is to show the following result.

\begin{proposition}
[Constants of quasi-geodicity]
\label{prop:qgdq_constants}
    Let us suppose that \(\tan(\pi/p) < \cos(\pi/q)\).
    There exist two constants \(k_{q,p} \ge 1\) and \(c_{q,p} \ge 0\) such that for all \(\displaystyle\widehat\gamma = \bigcup_{i\in\mathbb{Z}} [u_i, u_{i+1}] \in \widehat{\Gamma}_\mathrm{red}(p,q)\) and \(f : \mathbb{R} \to \Hyp{2}\) the parametrisation of \(\widehat\gamma\) such that for all \(i \in \mathbb Z\), \(f(i\ell_{q, p}) = u_i\) and \(f\) is parametrised by arc length on \([i\ell_{q, p}, (i+1)\ell_{q, p}]\), we have
    \begin{equation}
    \label{eq:qgdq_constants}
        \forall s, t\in\mathbb R,
        \quad
        \frac{\abs{t - s}}{k_{q,p}} - c_{q,p}
        \le
        \hyperbolicDistance{f(s)}{f(t)}
        \le
        \abs{t - s}
    \end{equation}
    In particular, \(f\) is a complete \((k_{q,p}, c_{q,p})\)-quasi-geodesic.
\end{proposition}

\begin{proof}
Let \(s,\, t \in \mathbb{R}\).
Let us suppose, without loss of generality, that \(s \le t\) and let us take \(i,\,j\in\mathbb Z\) such that \(i\ell_{q,p}\le s\le(i+1)\ell_{q,p}\) and \((j-1)\ell_{q,p}\le t\le j\ell_{q,p}\).

\begin{figure}[ht]
    \centering
    \begin{tikzpicture}
        \coordinate (s) at (.5,.5);
        \coordinate (t) at (8.5,.5);

        \draw [dashed] (0,1)--(s);
        \draw (s)--(1,0)--(2,1)--(3,0)--(4,1);
        \draw [dashed] (4,1)--(5,0);
        \draw (5,0)--(6,1)--(7,0)--(8,1)--(t);
        \draw [dashed] (t)--(9,0);

        \draw [
            decorate,
            decoration={amplitude=5pt,brace,mirror},
            red
        ] (0,1)--(.5,.5) node [anchor=north east,red]{\(\ell_{q,p}\)}--(1,0);
        \draw [
            decorate,
            decoration={amplitude=5pt,brace},
            red
        ] (8,1)--(8.5,.5) node [anchor=south west,red]{\(\ell_{q,p}\)}--(9,0);

        \filldraw [black] (0,1) circle (1pt) node[anchor=south]{\(i \ell_{q,p}\)};
        \filldraw [black] (s) circle (1pt) node[anchor=south west]{\(s\)};
        \filldraw [black] (1,0) circle (1pt) node[anchor=north]{\((i+1) \ell_{q,p}\)};
        \filldraw [black] (8,1) circle (1pt) node[anchor=south]{\((j-1) \ell_{q,p}\)};
        \filldraw [black] (t) circle (1pt) node[anchor=north east]{\(t\)};
        \filldraw [black] (9,0) circle (1pt) node[anchor=north]{\(j \ell_{q,p}\)};
    \end{tikzpicture}
\end{figure}

\paragraph*{Upper bound.}

The upper bound in \eqref{eq:qgdq_constants} is due to the triangular inequality.
Indeed, we have
\begin{align}
    \hyperbolicDistance{f(s)}{f(t)}
    &
    \le
    \hyperbolicDistance{f(s)}{u_{i+1}}
    +
    \hyperbolicDistance{u_{i+1}}{u_{j-1}}
    +
    \hyperbolicDistance{u_{j-1}}{f(t)}
    \\
    &
    \le
    [(i+1) \ell_{q,p} - s]
    +
    [(j - i - 2) \ell_{q,p}]
    +
    [t - (j-1) \ell_{q,p}]
    \\
    &
    \le
    \abs{t - s}.
\end{align}

\paragraph*{Lower bound.}

Using the reverse triangular inequality and because the distances between the points \(u_i\) and \(f_{s}\) and the points \(u_j\) and \(f_t\) is less than \(\ell_{q, p}\), we can observe that, if \(\tan(\pi/p)<\cos(\pi/q)\), then
\begin{equation}
\label{eq:low_bou_1}
    \hyperbolicDistance{f(s)}{f(t)}
    \ge
    \hyperbolicDistance{u_i}{u_j}
    -
    2 \ell_{q, p}
    \ge
    \hyperbolicDistance{v_i}{v_{j}}
    -
    2 \ell_{q,p}
    =
    \hyperbolicDistance{v_0}{v_{j-i}}
    -
    2 \ell_{q,p},
\end{equation}
where the second inequality is a consequence of the discrete hyperbolic Schur theorem \ref{lem:schur}, whose hypotheses are verified using \eqref{eq:alphaqp_inequality}, for the sequence \((v_i)_{i\in\mathbb{Z}}\) of points of \(\Hyp{2}\) such that for all \(i \in \mathbb{Z}\), \(\hyperbolicDistance{v_i}{v_{i+1}} = \ell_{q,p}\) and \(\angle(v_{i-1},v_i,v_{i+1})=A_p\) with \(v_0\in\mathbb{R}_-\) and \(v_1=0\).

Let us denote \(r_{q,p} = \tanh(\ell_{q, p} / 2)\) the Euclidean distance between \(0\) and a point at distance \(\ell_{q, p}\) from \(0\).
Let us consider the isometry \(\phi_{q, p}\) of \(\Hyp{2}\) (acting on \(\mathbb{D}\)) such that \(\phi_{q, p}(-r_{q,p})=0\) and \(\phi_{q, p}(0)=r_{q,p}e^{\bfi a_p}\) (where \(\bfi\) is the unit complex of argument \(\displaystyle \frac{\pi}{2}\)).
In that case, for all \(i\in\mathbb{Z}\), \(v_i = \phi_{q, p}^i(v_0) = \phi_{q, p}^{i-1}(v_1)\) and \(\phi_{q, p}\) is represented by the matrix
\begin{equation}
    M_{q, p} =
    \begin{pmatrix}
        \mathrm{e}^{\bfi a_p / 2} & r_{q,p} \mathrm{e}^{\bfi a_p / 2} \\
        r_{q,p} \mathrm{e}^{-\bfi a_p / 2} & \mathrm{e}^{-\bfi a_p / 2}
    \end{pmatrix}.
\end{equation}
The eigenvalues \(\lambda_{\pm}\) and eigenvectors \(v_{\pm}\) of \(M_{q, p}\) are
\begin{equation}
    \lambda_\pm
    =
    \cos(a_p/2)
    \pm
    \delta
    \in
    \mathbb{R}
    \quad
    \text{and}
    \quad
    v_\pm
    =
    \begin{pmatrix}
        x_\pm\\
        1
    \end{pmatrix} \; ,
\end{equation}
where
\(
    \displaystyle
    x_{\pm}
    =
    \mathrm{e}^{\bfi a_p/2}
    \frac{
        \bfi \sin(a_p/2)
        \pm
        \delta
    }{
        {r_{q,p}}
    }
\)
and
\(\delta^2 = r_{q,p}^2 - \sin^2(a_p/2) > 0\) as soon as \(\tan(\pi/p) < \cos(\pi/q)\).
Hence, by working in the basis \(\{v_+,v_- \}\), we obtain for \(n\in\mathbb N\),  
\begin{equation}
v_{n+1} = \phi_{q, p}^n(v_1) = x_+ \frac{\rho^n - 1}{\rho^n - \xi},
\end{equation}
where
\(\displaystyle \rho = \frac{\lambda_+}{\lambda_-}\)
and
\(\displaystyle \xi = \frac{x_+}{x_-}\).

By applying the function \(2 \artanh\) to the inequality obtained in the following equation
\begin{equation}
    \abs{v_{n+1}}
    =
    \frac{
        \rho^n - 1
    }{
        \abs{\rho^n - \xi}
    }
    =
    \frac{
        \rho^n - 1
    }{
        \sqrt{
            \left(\rho^n - \rePart{\xi}\right)^2 + \imPart{\xi}^2
        }
    }
    =
    \frac{
        \rho^n - 1
    }{
        \sqrt{
            \rho^{2n} - 2 \rePart{\xi} \rho^n + 1
        }
    }
    \ge
    \frac{
        \rho^n - 1
    }{
        \rho^n + 1
    },
\end{equation}
we have
\begin{equation}
\label{eq:low_bou_2}
    \hyperbolicDistance{v_0}{v_n}
    =
    \hyperbolicDistance{v_1}{v_{n+1}}
    =
    2 \artanh \abs{v_{n+1}}
    \ge
    n \log \rho.
\end{equation}

If we set \(\displaystyle k_{q,p} = \frac{\ell_{q, p}}{\log\rho} > 0\), we obtain using \eqref{eq:low_bou_1} and \eqref{eq:low_bou_2},
\begin{equation}
    \hyperbolicDistance{f(s)}{f(t)}
    \ge
    \frac{j-i}{k_{q,p}} \ell_{q,p}
    -
    2 \ell_{q,p}
    \ge
    \frac{\abs{t-s}}{k_{q,p}}
    -
    2 \ell_{q,p}.
\end{equation}
We conclude by observing that
\begin{equation}
    k_{q,p}
    =
    \frac{
        \arcosh\left(
            \frac{
                \cos(\pi/q)
            }{
                \sin(\pi/p)
            }
        \right)
    }{
        \arcosh\left(
            \frac{
                \cos(\pi/q)
                \cos(a_p/2)
            }{
                \sin(\pi/p)
            }
        \right)
    }
    \ge
    \frac{
        \arcosh\left(
            \frac{
                \cos(\pi/q)
            }{
                \sin(\pi/p)
            }
        \right)
    }{
        \arcosh\left(
            \frac{
                \cos(\pi/q)
            }{
                \tan(\pi/p)
            }
        \right)
    }
    \ge
    1
\end{equation}
and by setting \(c_{q,p} = 2 \ell_{q, p} \ge 0\).
\end{proof}

\begin{remark}
The proof can be extended to any fixed length \(\ell\) and any angle \(\theta\) with \(\tanh(\ell/2) > \sin(\theta/2)\).
\end{remark}

\subsection{Hausdorff distance estimates}
\label{ssec:haus_dist_esti}

\begin{lemma}
[Uncountably many geodesics are close enough from zigzag paths]
\label{lem:good_geod}
    For all \(p\ge5\), there exists \(q_p\ge5\) such that for all \(q\ge q_p\), there exists uncountably many \(\gamma\in\Gamma\) such that
    \begin{equation}
    \label{eq:proximity}
        \hausdorffDistance{\gamma}{\widehat{\gamma}}
        <
        \ell_{p,q} / 2
        \quad
        \text{for some \(\widehat\gamma\in\widehat\Gamma_r(p,q)\)}
        .
    \end{equation}
\end{lemma}

\begin{proof}
Let \(p \ge 5\), and let \(q\ge3\) such that \(\tan(\pi/p)>\cos(\pi/q)\).
Let us take \(\displaystyle\widehat{\gamma}=\bigcup_{i\in\mathbb{Z}} [u_i, u_{i+1}]\in\widehat{\Gamma}_\mathrm{red}(p,q)\).

% \paragraph*{Ends of zigzag path.}

Thanks to Proposition \ref{prop:qgdq_constants}, we know that \(\widehat{\gamma}\) is a complete quasi-geodesic, and thus there exists \(\gamma\) a complete geodesic such that \(\hausdorffDistance{\gamma}{\widehat{\gamma}} < \infty\).
Let us denote by \(\vartheta_1\) and \(\vartheta_2\) both ends of \(\gamma\) in \(\partial\Hyp{2}\), \ie~\(\gamma=[\vartheta_1,\vartheta_2]\).

% \paragraph*{Bound on \(q\).}

Thanks to Proposition \ref{prop:qgdq_constants}, we know that there exists \(f:\mathbb{R}\to\Hyp{2}\) a complete \((k_{q,p}, c_{q,p})\)-quasi-geodesic such that \(f(\mathbb{R}) = \widehat{\gamma}\) with
\begin{equation}
k_{q,p}
=
\frac{
    \arcosh\left(
        \frac{
            \cos(\pi/q)
        }{
            \sin(\pi/p)
        }
    \right)
}{
    \arcosh\left(
        \frac{
            \cos(\pi/q)
            \cos(a_p/2)
        }{
            \sin(\pi/p)
        }
    \right)
}
=
\Theta_{q}(1)
\quad
\text{and}
\quad
c_{q,p}
=
4
\arcosh\left(
    \frac{
        \cos(\pi/q)
    }{
        \sin(\pi/p)
    }
\right)
=
\Theta_{q}(1).
\end{equation}
Using the Morse--Mostow lemma \ref{lem:Morse--Mostow}, we obtain that
\begin{equation}
\hausdorffDistance{\widehat{\gamma}}{[\vartheta_1, \vartheta_2]}
\le
C k_{q,p}^2 (c_{q,p} + \delta)
=
O_q(1).
\end{equation}
Moreover,
\begin{equation}
\frac{
    \ell_{p, q}
}{
    2
}
=
\arcosh\left(
    \frac{
        \cos(\pi/p)
    }{
        \sin(\pi/q)
    }
\right)
\xrightarrow[q \to \infty]{}
\infty,
\end{equation}
as such there exists \(q_p\in\mathbb{N}\) such that for \(q \ge{q_p}\), \eqref{eq:proximity} holds.

% \paragraph*{Uncountably many geodesics.}

For the rest of the proof, let us suppose that \(q\ge q_p\).
As there exists uncountably many \(\widehat\gamma\), we only have to prove that all \(\gamma\) at finite distance from a \(\widehat\gamma\) are different.
Let us take \(\widehat\gamma_1\) and \(\widehat\gamma_2\) two distinct elements of \(\widehat\Gamma_r(p,q)\).
Let us denote by \(\gamma_1\) (\resp~\(\gamma_2\)) the geodesic of \(\Hyp{2}\) at distance smaller than \(\ell_{p,q}/2\) from \(\widehat\gamma_1\) (\resp~\(\widehat\gamma_2\)).
As \(\widehat\gamma_1\neq\widehat\gamma_2\), there exists a face of the lattice which is crossed by only one of both zigzag paths.
As for all \(i\in\left\{1,2\right\}\), \(\hausdorffDistance{\widehat\gamma_i}{\gamma_i}<\ell_{p,q}/2\), only one of both geodesics crosses that face.
Hence, \(\gamma_1 \neq \gamma_2\).
\end{proof}

\subsection{Proof of theorem \ref{thm.principal}}
\label{ssec:proof_main_theorem}

Let \(p\ge5\) and \(q\ge q_p\).

% \paragraph*{Well definition of the boundary condition of the zigzag states.}

Let us take \(\displaystyle\widehat{\gamma}=\bigcup_{i\in\mathbb{Z}} [u_i, u_{i+1}]\in\widehat{\Gamma}_\mathrm{red}(p,q)\).
For \(i\in\mathbb{Z}\), there exists exactly two horocycles going through \(u_i\) and \(u_{i+1}\).
Let us denote by \(H_i\) the union of both open disks described by the horocycles and by \(U^-_i\) (\resp~\(U^+_i\)) the connected component of \(\mathbb{D} \setminus H_i\) that contains \(u_i\) (\resp~\(u_{i+1}\)).
In particular, for all \(i\in\mathbb{Z}\), \([u_i,u_{i+1}] \subset \overline{H_i}\).
Using the Lemma of Subsection 2.5.A of \cite{FG97}, for all \(i\in\mathbb{Z}\), \(U^-_{i-1} \subset U^-_i\) and \(U^+_{i+1} \subset U^+_i\).
As such we conclude that \(\widehat{\gamma}\) is a simple curve of infinite length.
We conclude that \(\Hyp{2}\setminus\widehat\gamma\) has two connected components and splits \(V_{p,q}\) in two sets of connected vertices.
As such, the configuration \(\eta^{\widehat{\gamma}}\) are well-defined.

% \paragraph*{Uncountably many \emph{extremal} Series--Sinai states.}

Thanks to Lemma \ref{lem:good_geod}, we know that for \(q \ge q_p\), there exists uncountably many \(\gamma\in\Gamma\) such that \(\hausdorffDistance{\gamma}{\widehat\gamma} < \ell_{p,q}/2\) for some \(\widehat\Gamma_r(p,q)\).
Then thanks to Lemma \ref{lem:link_SS_DACLN}, for \(\mu^{\gamma}_{\beta}\) a weak limit of finite volume states with boundary condition \(\eta^\gamma\),
\(
    \mu_\beta^\gamma
    \in
    \exG_\beta
\)
at low enough temperature.

As the Series--Sinai states are mutually singular at low enough temperature, and as such pairwise distinct, we proved that there exist uncountably many \emph{extremal} Series--Sinai states for the ferromagnetic nearest-neighbour Ising model on \(\Lpq\) at low enough temperature.

\vfill

\paragraph*{Acknowledgements.}

We thank François Dahmani and Pierre Will for the enlightening discussions.
We thank Loren Coquille and Matteo D'Achille for their many helpful suggestions, proofreading and constructive feedback.
We are grateful to the Institut Élie Cartan de Lorraine for excellent working conditions in the occasion of an invitation (March 2026).
We acknowledge the support of ANR Coconut (\texttt{ANR-25-CE40-1147}).

\newpage

\printbibliography

\end{document}